\documentclass[reqno,11pt,a4paper]{amsart}

\usepackage{color, varwidth}

\usepackage{hyperref}

\usepackage{amsmath, amsthm, amsfonts, amssymb}

\usepackage{mathtools}
\mathtoolsset{showonlyrefs}

\usepackage[applemac]{inputenc}

\theoremstyle{plain}
\newtheorem{theorem}{Theorem}
\newtheorem{proposition}[theorem]{Proposition}
\newtheorem{lemma}[theorem]{Lemma}
\newtheorem{corollary}[theorem]{Corollary}

\theoremstyle{definition}

\newtheorem{definition}[theorem]{Definition}

\theoremstyle{remark}
\newtheorem{remark}[theorem]{Remark}

\DeclareMathOperator{\im}{Im}
\def\Z{\mathbb{Z}}	
	
\def\R{\mathbb{R}}	
\renewcommand{\leq}{\leqslant} 		
\renewcommand{\geq}{\geqslant}

\def\mlambda{\boldsymbol{\lambda}}
\def\mmu{\boldsymbol{\mu}}
\def\md{\boldsymbol{\d}}
\DeclareMathOperator{\Span}{Span}
\def\ud{\mathrm{d}}

\def\d{\partial}

\let\ker\relax
\DeclareMathOperator{\ker}{Ker}
\def\cA{\mathcal{A}}
\def\hcA{\hat{\cA}}
\def\cF{\mathcal{F}}
\def\hcF{\hat{\cF}}

\def\cV{\mathcal{V}}
\DeclareMathOperator{\der}{Der}

\begin{document}

\title[Cohomology of Dubrovin-Novikov brackets]{Poisson cohomology of scalar multidimensional Dubrovin-Novikov brackets}
\author{Guido Carlet}
\author{Matteo Casati}
\author{Sergey Shadrin}
\begin{abstract}
We compute the Poisson cohomology of a scalar Poisson bracket of Dubrovin-Novikov type with $D$ independent variables. We find that the second and third cohomology groups are generically non-vanishing in $D>1$. Hence, in contrast with the $D=1$ case, the deformation theory in the multivariable case is non-trivial.
\end{abstract}
\maketitle

\tableofcontents

\section{Introduction}

The multidimensional Dubrovin-Novikov (DN) type Poisson brackets were introduced by Dubrovin and Novikov in~\cite{dn83, dn84}. 

Let $x=(x^1, \dots , x^D)$ be coordinates on the torus $T^D$ and $u=(u^1, \dots , u^N)$ be variables on an open ball $U \subset \R^N$ (or more generally local coordinates on a smooth $N$-dimensional manifold $M$). The Dubrovin-Novikov brackets are of the form
\begin{equation} \label{DN}
\{ u^i(x) , u^j(y) \} = \sum_{\alpha=1}^D \left( g^{ij \alpha} (u(x)) \partial_{x^\alpha} \delta (x-y) + b_k^{ij \alpha}(u(x)) \partial_{x^\alpha} u^k (x) \delta(x-y) \right)
\end{equation}
were $g^{ij \alpha}(u)$ and $b_k^{ij\alpha}(u)$ are smooth functions over $U$, and $\delta(x-y)$ denotes the multidimensional Dirac delta function
\begin{equation}
\delta(x-y) = \delta(x^1 - y^1) \cdots \delta(x^D - y^D) .
\end{equation}

The fact that the brackets~\eqref{DN} are Poisson (i.e. that they are skew-symmetric and satisfy the Jacobi identity) imposes several conditions on the functions $g^{ij \alpha}(u)$ and $b_k^{ij\alpha}(u)$. 

Such conditions have been studied and are still being studied for different values of $D$, $N$ by several authors. 
In particular they have been related to certain geometric structures over $M$, since the seminal paper by Dubrovin and Novikov~\cite{dn83} which states a one to one  correspondence between the nondegenerate $D=1$ Poisson brackets and flat contravariant pseudo-Riemannian metrics. 
A general set of equations valid for all $D$, $N$ has been obtained by Mokhov~\cite{m88}; a classification of the nondegenerate brackets for $D=2$ exists for $N=2$ \cite{m08} and $N=3,4$ \cite{fls15}, where it relies on the notion of Killing $(1,1)$-tensors; in special cases a classification can be obtained for $D=2$ and arbitrary $N$, or for $D=3$, $N\leq3$ \cite{fls15}, or again for arbitrary $D$ and $N$ \cite{m08}. Very recently some first results on the classification of degenerate brackets have appeared \cite{s14,s15}.

As pointed out in~\cite{dz01}, at least in the scalar case, an important problem is to classify the dispersive deformations of such brackets. 
Let us state this problem precisely in the multidimensional case. 

Let $\cA$ be the space of differential polynomials, i.e. formal power series in the variables $\partial_{x^1}^{k_1} \cdots \partial_{x^D}^{k_D} u^i$ with coefficients which are smooth functions of $u^i$:
\begin{equation}
\cA = C^\infty (U) [[ \{ \partial_{x^1}^{k_1} \cdots \partial_{x^D}^{k_D} u^i \text{ with } k_1, \dots ,k_D \geq0 , \,(k_1, \dots  , k_D ) \not= 0 \} ]].
\end{equation} 
The standard degree $\deg$ on $\cA$ counts the number of derivatives $\partial_{x^i}$ in a monomial, i.e., it is defined by $\deg ( \partial_{x^1}^{k_1} \cdots \partial_{x^D}^{k_D} u^i) = k_1 + \cdots + k_D$.

We consider dispersive deformations of multidimensional DN brackets of the form
\begin{align} \label{def}
\{ u^i(x) ,u^j(y) \}^\epsilon &= \{ u^i(x) , u^j(y) \} +\\ &+\sum_{k>0} \epsilon^k \sum_{\substack{k_1,\dots ,k_D \geq0 \\ k_1 + \cdots +k_D \leq k+1}}  A^{ij}_{k; k_1, \dots,k_D}(u(x)) \partial_{x^1}^{k_1} \cdots \partial_{x^D}^{k_D} \delta(x-y)  \notag
\end{align}
where $A^{ij}_{k; k_1, \dots,k_D} \in \cA$ and $\deg A^{ij}_{k; k_1, \dots,k_D} = k - k_1 \cdots - k_D +1$.

The Miura-type transformations (of the second kind~\cite{lz11}) are changes of variables of the form
\begin{equation}
v^i = u^i + \sum_{k\geq1} \epsilon^k F_k^i
\end{equation}
where $F_k^i \in \cA$ and $\deg F_k^i = k$. They form a group called Miura group.

The main problem is to classify the dispersive deformations~\eqref{def} of the multidimensional DN type brackets $\{,\}$ up to equivalence. Two deformations are equivalent when they are related by a Miura transformation. A deformation $\{ , \}^\epsilon$ is called trivial when it is equivalent to the undeformed DN type brackets $\{, \}$.

It is a well known general fact that the dispersive deformations are governed by the second and third Poisson cohomology groups associated with the dispersionless Poisson bracket $\{,\}$. In particular the second Poisson cohomology classifies the infinitesimal deformations, while the third Poisson cohomology encodes the obstructions to the extension of an infinitesimal deformation to a full dispersive deformation. See Section~\ref{2} for the relevant definitions. 

In the case of a single independent variable, $D=1$, the Poisson cohomology of a DN type Poisson bracket has been shown to vanish, in positive degree, by Getzler~\cite{g02}. The main consequence of this result is that all deformations of one-dimensional DN type Poisson brackets are trivial, i.e., they are Miura equivalent to their dispersionless limit. Independent proofs of the triviality of the deformation problems in the case $D=1$ have been obtained within different frameworks \cite{dms05,dz01,dsk14}. Moreover, some first results for $D=2$ have been published by one of the authors \cite{c14}.

In the scalar case $N=1$ the general form of a multidimensional DN type  Poisson bracket is~\cite{m88}
\begin{equation}  \label{scalpb}
\{ u(x) , u(y) \} = g(u(x)) c^i \frac{\partial }{\partial x^i} \delta(x-y) + \frac12 g'(u(x)) c^i \frac{\partial u}{\partial x^i}(x) \delta(x-y) 
\end{equation}
where $g(u)$ is a non-vanishing function and $c^i$ are constants, with $i=1, \dots , D$.

Our main result is the computation of the full Poisson cohomology of the Poisson bracket~\eqref{scalpb}, in a quite implicit form, see Theorem~\ref{mainthm}. 
As a consequence of this result, we obtain the cohomology groups of low degree, which are relevant for the deformation theory, for some values of $D$, see Section~\ref{3}. 
We find in particular that for $D>1$ the second and third Poisson cohomology groups are generically non-vanishing. 
Therefore (formal) infinitesimal deformations of~\eqref{scalpb} are still parametrized by a class in $H^2(\hcF)$, although this class is in general non-homogeneous in the standard degree, but no general statement can be made about the possibility of extending such infinitesimal deformation to a full deformation of the form~\eqref{def}, even in the homogeneous case. See Section~\ref{27} for further details. 

Interesting problems to consider at this stage would be: the study of how particular examples of full dispersive Poisson brackets of DN type in $D>1$ fit in this scenario; and the formulation of classification and existence theorems for certain subclasses of homogeneous deformations.

The paper is organized as follows: in Section 2 we introduce the main tools of formal multidimensional variational calculus. In Section 3 we specialise to the case of a single dependent variable, and we prove our main Theorem. In Section 4 we explicitly compute a few homogeneous components of the first and second cohomology groups, using the formalism of Poisson Vertex Algebras.

\subsubsection*{Acknowledgment} G.~C.~and S.~S.~were supported by the Netherlands Organization for Scientific Research. G.~C.~and M.~C.~were partly supported by Young Researcher Project 2014 ``Geometric and analytic aspects of integrable systems'' by the Italian Institute for Higher Mathematics, Group of Mathematical Physics.

\section{Functional variational calculus and deformations in the multidimensional case}
\label{2}

In this Section we introduce the basic notions of local multivectors,  Schouten-Nijenhuis brackets, $\theta$ formalism and Miura transformations in the case of $D$ independent and $N$ dependent variables. In \S 2.1-2.5 we give the $D\geq1$ version of some basic constructions of~\cite{lz11, lz13}, omitting the generalisations of most of the proofs, which will be addressed in a subsequent publication. In \S 2.6 we introduce short sequences related with partial integrations, which are a peculiar feature of the $D>1$ case, and we prove that they are exact, adapting the proof of exactness of the variational bicomplex~\cite{anderson}. In \S 2.7 we consider the relation between deformation theory and Poisson cohomology.

\subsection{Multi-index notation}
Let us introduce the multi-index notation as follows. Denote by $\xi_1, \dots \xi_D$ the standard basis of the semiring $Z:=\Z_{\geq0}^D$. Let a multi-index $S=\sum_{i=1}^D s_i \xi_i$ be an arbitrary element of $Z$. The length of the multi-index $S$ is $|S| = \sum_{i=1}^D s_i$. 
We denote by $Z_j \subset Z$ the set of multi-indices $S=\sum_{i=1}^D s_i \xi_i$ with $s_j =0$. 
We denote by $\partial^S$ the operator $\partial_{x^1}^{s_1} \cdots \partial_{x^D}^{s_D}$. 
Sums over repeated indices and multi-indices are assumed.

\subsection{Local multivectors}

Using the multi-index notation, the space of differential polynomials is
\begin{equation}
\cA = C^\infty (U) [[ \{ u^{\alpha, S} , \alpha=1,\dots,N, |S|>0 \} ]],
\end{equation}
where we denote $u^{\alpha,S} = \partial^S u^\alpha$. 
The standard gradation $\deg$ on $\cA$ is given by $\deg u^{\alpha, S} = |S|$. We denote $\cA_d$ the homogeneous component of degree $d$. 

On $\cA$ we define commuting derivations $\partial_{x^i}$ for $i=1, \dots , D$ by 
\begin{equation}
\partial_{x^i} = \sum_{\alpha, S} u^{\alpha, S+\xi_i} \frac{\partial }{\partial u^{\alpha, S} }.
\end{equation}

The elements of the quotient
\begin{equation}
\cF = \frac{\cA}{\partial_{x^1} \cA + \cdots + \partial_{x^D}\cA}
\end{equation}
are called local functionals. 
The standard gradation on $\cA$ induces a standard gradation on $\cF$, since the operators $\partial_{x^i}$ are homogeneous. 
We denote the projection map from $\cA$ to $\cF$ as a multiple integral, which associates to $f\in\cA$ the element
\begin{equation}
\int^D f \ d^Dx
\end{equation}
in $\cF$.

The variational derivative of  a local functional $F=\int^D f \ d^Dx$ is defined as 
\begin{equation} \label{varu}
\frac{\delta F}{\delta u^\alpha} = \sum_S (-1)^{|S|} \partial^S \frac{\partial f}{\partial u^{\alpha,S}} .
\end{equation}
One can easily prove that
\begin{equation} \label{commupa}
\left[ \frac{\partial }{\partial u^{\alpha,S}}, \partial_{x^j} \right] = \frac{\partial }{\partial u^{\alpha,S-\xi_j}}, \text{ if $S=\sum_i s_i \xi_i$ with $s_j>0$},
\end{equation}
and is equal to zero otherwise, and
\begin{equation} \label{varpa}
\frac{\delta}{\delta u^\alpha} \partial_{x^i} =0 .  
\end{equation}
From this in particular it follows that the variational derivative of $\int^D f \ d^Dx$ does not depend on the choice of the density $f$.

A local $p$-vector $P$ is a linear $p$-alternating map from $\cF$ to itself of the form
\begin{equation} \label{pvect}
P(I_1, \dots ,I_p) = \int^D P^{\alpha_1 , \dots , \alpha_p}_{S_1, \dots , S_p} \ \partial^{S_1}  \left( \frac{\delta I_1}{\delta u^{\alpha_1}} \right) \cdots  \partial^{S_p} \left( \frac{\delta I_p}{\delta u^{\alpha_p}} \right)
 \ d^Dx
\end{equation}
where $P^{\alpha_1 , \dots , \alpha_p}_{S_1, \dots , S_p}  \in \cA$, for arbitrary $I_1, \dots, I_p \in \cF$. We denote the space of local $p$-vectors by $\Lambda^p \subset \mathrm{Alt}^p(\cF, \cF)$.


The following Lemma generalizes Lemma 2.1.7 of~\cite{lz11}.\begin{lemma} \label{217}
Let $P^\alpha \in \cA$. If 
\begin{equation}
\int^D \sum_\alpha P^\alpha \frac{\delta I}{\delta u^\alpha} \ d^Dx =0
\end{equation}
for all $I \in \cA$, then $P^\alpha = \sum_{i=1}^D c_i u^{\alpha, \xi_i}$ for some $c_i \in \R$. 
\end{lemma}

\subsection{The $\theta$ formalism}

Let $\hcA$ be the algebra of formal power series in the commutative variables $u^{\alpha,S}$, $|S|>0$ and anticommutative variables 
$\theta_\alpha^S$, $|S|\geq0$ with coefficients given by smooth functions on $U$, i.e.,
\begin{equation}
\hcA := C^\infty (U) [[ \left\{ u^{\alpha,S} , |S|>0\right\}\cup \left\{ \theta_\alpha^S, |S|\geq0 \right\}]] .
\end{equation}
The standard gradation $\deg$ and the super gradation $\deg_\theta$ of $\hcA$ are defined by setting
\begin{equation}
\deg u^{\alpha,S} =\deg \theta_\alpha^S = |S|, \qquad 
\deg_\theta u^{\alpha,S} = 0 , \qquad \deg_\theta \theta_\alpha^S = 1 .
\end{equation}
We denote $\hcA_d$, resp. $\hcA^p$, the homogeneous components of standard degree $d$, resp. super degree $p$, while $\hcA_d^p := \hcA_d \cap \hcA^p$. Clearly $\hcA^0 =\cA$.

The commuting derivations $\partial_{x^i}$ for $i=1, \dots , D$ are extended to $\hcA$ by 
\begin{equation}
\partial_{x^i} = \sum_{\alpha, S} \left( u^{\alpha, S+\xi_i} \frac{\partial }{\partial u^{\alpha, S} } + \theta_\alpha^{S+\xi_i} \frac{\partial }{\partial \theta_\alpha^S} \right).
\end{equation}
Note that the kernel of $\partial_{x^i}$ on $\hcA$ is $\R$.

We denote by $\hcF$ the quotient of $\hcA$ by the subspace $\partial_{x^1} \hcA + \cdots + \partial_{x^D} \hcA$, and by a multiple integral $\int^D \cdot \ d^Dx$ the projection map from $\hcA$ to $\hcF$. Since the derivations $\partial_{x^i}$ are homogeneous, i.e., $\deg \partial_{x^i} =1$ and $\deg_\theta \partial_{x^i} =0$, $\hcF$ inherits both gradations of $\hcA$.  

Equations~\eqref{commupa} and~\eqref{varpa} hold and, similarly,
\begin{equation}\label{commupatheta}
\left[ \frac{\partial }{\partial \theta_{\alpha}^{S}}, \partial_{x^j} \right] = \frac{\partial }{\partial \theta_{\alpha}^{S-\xi_j}}, \text{ if $S=\sum_i s_i \xi_i$ with $s_j>0$},
\end{equation}
and is equal to zero otherwise. It follows that the variational derivative
\begin{equation} \label{varde}
\frac{\delta}{\delta \theta_\alpha} = \sum_{\alpha, S} (-1)^{|S|} \partial^{S} \frac{\partial }{\partial \theta_\alpha^S} 
\end{equation}
satisfies
\begin{equation}\label{varderpa}
\frac{\delta}{\delta \theta_\alpha} \partial_{x^i} = 0.
\end{equation}
Hence both variational derivatives~\eqref{varpa} and~\eqref{varde} define  maps from $\hcF$ to $\hcA$.

\begin{proposition}
The space of local multi-vectors $\Lambda^p$ is isomorphic to $\hcF^p$ for $p\not=1$. Moreover
\begin{equation}
\Lambda^1 
\cong \frac{\hcF^1}{\oplus_i \R \int u^{\alpha, \xi_i} \theta_\alpha} 
\cong \frac{\der'(\cA)}{\oplus_i \R \partial_{x^i}},
\end{equation}
where $\der'(\cA)$ denotes the space of derivations of $\cA$ that commute with $\partial_{x^i}$, for $i=1,\dots ,D$, and $\der'(\cA) \cong \hcF^1$.
\end{proposition}
\begin{remark} 
Let us give here some remark on the proof of this Proposition, following the argument of~\cite{lz11}. 
Notice that for $p=0$, the isomorphism is trivial, since $\hcF^0 = \cF = \Lambda^0$. Let us assume instead that $p\geq1$. 
Given $P\in\hcF^p$, and arbitrary $I_1, \dots, I_p \in \cF$, let
\begin{equation}\label{defIota}
\iota (P)(I_1, \dots ,I_p) = \frac{\partial }{\partial \theta_{\alpha_p}^{S_p} }\cdots \frac{\partial }{\partial \theta_{\alpha_1}^{S_1}} P \cdot
 \partial^{S_1} \left(\frac{\delta I_1}{\delta u^{\alpha_1}}\right) \cdots \partial^{S_p}\left( \frac{\delta I_p}{\delta u^{\alpha_p}} \right).
\end{equation}
Clearly $\iota(P)$ is an $p$-alternating map from $\cF$ to $\cA$, and it satisfies
\begin{equation}
\iota(\partial_{x^i} P)(I_1, \dots , I_p) = \partial_{x^i} (\iota(P)(I_1, \dots , I_p) ) .
\end{equation}
The desired map $\tilde{\iota}$ from $\hcF^p$ to $\Lambda^p$ is  then defined by 
\begin{equation}
\tilde{\iota}\left( \int^D P \ d^Dx\right) = \int^D \iota(P) \ d^Dx .
\end{equation}
Surjectivity of $\tilde{\iota}$ is easy to see; indeed the local $p$-vector~\eqref{pvect} is the image through $\iota$ of
\begin{equation}
P = \frac1{p!}  P^{\alpha_1 , \dots , \alpha_p}_{S_1, \dots , S_p} \theta_{\alpha_1}^{S_1} \cdots \theta_{\alpha_p}^{S_p}.
\end{equation}
In the case $p=1$, an element $P$ of $\hcF^1$ can be written uniquely as
\begin{equation}
P = \int^D P^\alpha \theta_\alpha \ d^Dx
\end{equation}
for $P^\alpha \in \cA$. The map that sends $P$ to 
\begin{equation}
\partial_P = \sum_S \partial^S P^\alpha \frac{\partial }{\partial u^{\alpha, S}} 
\end{equation}
defines an isomorphism $\hcF^1 \cong \der'(\cA)$. Notice that 
\begin{equation}
\tilde{\iota}(P)(I) = \int^D \partial_P (I) \ d^Dx. 
\end{equation}
Clearly, for each $i$, the derivation associated with $P=\int^D \sum_\alpha u^{\alpha, \xi_i} \theta_\alpha \ d^Dx$ corresponds to $\partial_P = \partial_{x^i}$, therefore is in the kernel of $\tilde{\iota}$. From Lemma~\ref{217} it follows that the kernel of $\tilde{\iota}$ is indeed generated by these elements. 

It remains to show that $\tilde{\iota}$ is injective for $p\geq 2$, which can be done essentially by adapting the proof given in~\cite{lz11} to the present case. Since this argument relies on further technical lemmas on differential operators we prefer to skip it in the context of this paper. 
\end{remark}

\subsection{The Schouten-Nijenhuis bracket}

The Schouten-Nijenhuis bracket 
\begin{equation}
[,]:\hcF^p \times \hcF^q \to \hcF^{p+q-1}
\end{equation}
is defined as  
\begin{equation}\label{defsch}
[P , Q ] = \int^D \left( \frac{\delta P}{\delta\theta_\alpha} \frac{\delta Q}{\delta u^\alpha} + (-1)^p \frac{\delta P}{\delta u^\alpha} \frac{\delta Q}{\delta \theta_\alpha} \right) \ d^Dx .
\end{equation}

It is a bilinear map that satisfies the graded symmetry 
\begin{equation}
[P,Q] = (-1)^{pq} [Q,P]
\end{equation}
and the graded Jacobi identity
\begin{equation}
(-1)^{pr} [[P,Q],R] + (-1)^{qp} [[Q,R],P] + (-1)^{rq} [[R,P],Q] =0 
\end{equation}
for arbitrary $P\in\hcF^p$, $Q\in\hcF^q$ and $r\in\hcF^r$.


A bivector $P\in\hcF^2$ is a Poisson structure when $[P,P]=0$. In such case $d_P := ad_P = [P, \cdot ]$ squares to zero, as a consequence of the graded Jacobi identity, and the cohomology of the complex $(\hcF, d_P)$ is called Poisson cohomology of $P$.

\subsection{The differential on $\hcA$}

Given an element $P \in \hcF^p$ we define the following differential operator on $\hcA$
\begin{equation}\label{eq:defDP}
D_P = \sum_S \left( \partial^S \left( \frac{\delta P}{\delta \theta_\alpha} \right) \frac{\partial }{\partial u^{\alpha,S}}  + (-1)^p \partial^S \left( \frac{\delta P}{\delta u^\alpha} \right) \frac{\partial }{\partial \theta_\alpha^S}  \right) .
\end{equation}
Since $[D_P, \partial_{x^i} ] =0$ for all $i=1, \dots , D$, the operator $D_P$ descends to an operator on $\hcF$ which is given by the adjoint action $ad_P = [P, \cdot]$ of $P$ on $\hcF$ via the Schouten-Nijenhuis bracket, i.e.,
\begin{equation}
ad_P \left( \int^D Q \ d^Dx  \right)= \int^D D_P (Q) \ d^Dx,
\end{equation}
for $Q \in \hcA$. This can be easily checked by integration by parts. 

\begin{remark}
It can be proved by an explicit computation that
\begin{equation} \label{morph}
D_{[P,Q]} = (-1)^{p+1} [D_P, D_Q] 
\end{equation}
which holds for $P\in\hcF^p$ and $Q\in\hcF^q$, where the brackets on the righthand-side represent the graded commutator that induces a graded Lie algebra structure on the space of graded derivations to which $D_P$ belongs, see~\cite{lz13}, i.e.,
\begin{equation}
[ D_{P} , D_{Q} ] = D_P D_Q -  (-1)^{(p-1)(q-1)} D_Q D_P .
\end{equation}
It follows that, if $P\in\hcF^2$ is such that $[P,P]=0$, then $D_P^2=0$; 
in other words, if $P\in\hcF^2$ is a Poisson structure,  then $D_P$ squares to zero, hence $(\hcA, D_P)$ is a differential complex.

We will not prove the identity~\eqref{morph} here, since in our specific case, where $P$ is given by formula~\eqref{deltaop},  the fact that $D_P^2=0$ simply follows from a trivial computation, see next Section. 
\end{remark}

\subsection{Partial integrations}

A crucial role in our construction is played by the following Lemma.
\begin{lemma} \label{lemma:kernel-di-zero}
The sequences
\begin{equation}\label{eq:Seq}
\begin{array}{rcccccccl}
0 & \to & \hat{\cA}/\R & \xrightarrow{\d_{x^1}} & \hat{\cA} & \xrightarrow{\int dx^1} & \hat{\cF}_1 &  \to & 0 \\
0 & \to & \hat{\cF}_1/\R & \xrightarrow{\d_{x^2}} & \hat{\cF}_1 & \xrightarrow{\int dx^2} & \hat{\cF}_2 &  \to & 0 \\
0 & \to & \hat{\cF}_2/\R & \xrightarrow{\d_{x^3}} & \hat{\cF}_2 & \xrightarrow{\int dx^3} & \hat{\cF}_3 &  \to & 0 \\
& & \vdots & & \vdots & & \vdots & & \\
0 & \to & \hat{\cF}_{D-1}/\R & \xrightarrow{\d_{x^D}} & \hat{\cF}_{D-1} & \xrightarrow{\int dx^D} & \hat{\cF}_D &  \to & 0 
\end{array}
\end{equation}
where
\begin{equation}
\hcF_i = \frac{\hcA}{\partial_{x^1} \hcA + \cdots + \partial_{x^i} \hcA }
\end{equation}
are exact.
\end{lemma}
\begin{proof}
The exactness of the first line is obvious: by quotienting out the kernel of $\partial_{x^1}$, which is given by constants, we have injectivity on the left side, while $\int dx^1$ just denotes the projection to the quotient $\hcF_1$, which is indeed surjective. 

Let us consider the second line. The derivation $\partial_{x^2}$ on $\hcA$ commutes with $\partial_{x^1}$, hence defines a map on $\hcF_1$ which we denote with the same symbol. Notice that 
\begin{equation}
\frac{\hcF_1}{\partial_{x^2} \hcF_1} = \frac{\hcA}{\partial_{x^1} \hcA + \partial_{x^2} \hcA} = \hcF_2,
\end{equation}
therefore surjectivity is guaranteed.

The same argument works for the rest of the lines, since it is easy to check that
\begin{equation}
\frac{\hcF_i}{\partial_{x^{i+1}} \hcF_i} = \frac{\hcA}{\partial_{x^1} \hcA +\cdots+ \partial_{x^{i+1}} \hcA} = \hcF_{i+1}. 
\end{equation}

It remains to be proved that the map induced by $\partial_{x^a}$ on $\hcF_{a-1}$ has kernel given by $\R$, for $a=2, \dots , D$. In order to do this we reformulate this property as the vanishing of the cohomology of some auxiliary complex, and then we construct an explicit homotopy contraction for that auxiliary complex that implies the vanishing of its cohomology.

The map induced by $\partial_{x^a}$ on $\hcF_{a-1}$ has only constants in the kernel if and only if for any function $f\in \hcA$ the conditions $\partial_{x^a} f \in \partial_{x^1}\hcA+\cdots+\partial_{x^{a-1}}\hcA$ and $f$ is non-constant imply that $f\in \partial_{x^1}\hcA+\cdots+\partial_{x^{a-1}}\hcA$. Consider the following complex:
\begin{equation}\label{eq:auxiliary-complex}
0\to\Omega^0\xrightarrow{d_H} \Omega^1 \xrightarrow{d_H} \cdots \xrightarrow{d_H} \Omega^{a-1}\xrightarrow{d_H} \Omega^a,
\end{equation}
where $\Omega^j$, $0\leq j\leq a$, is the space of local differential $j$-forms with coefficients in $\hcA/\mathbb{R}$, that is,
\begin{equation}
\Omega^j:= \bigoplus_{\substack {I \subset \{1,\dots,a\}, |I|=j \\ I=\{i_1<\dots<i_j\}}} \hcA/\mathbb{R} \cdot dx^{i_1}\wedge\cdots \wedge dx^{i_j},
\end{equation}
and the differential $d_H$ is equal to $\sum_{i=1}^a dx^i\wedge \partial_{x^i}$. In terms of this complex, the condition $\partial_{x^a} f = \partial_{x^1}g_1+\cdots+\partial_{x^{a-1}}g_{a-1}$ is the same as $d_H\omega =0$, where $\omega\in\Omega^{a-1}$ is given by
\begin{equation}
\omega:=\left[f\frac{\d}{\d (dx^a)}-\sum_{i=1}^{a-1}  g_i  \frac{\d}{\d (dx^i)}\right] dx^1\wedge\cdots \wedge dx^a.
\end{equation}
The property that $f\in \partial_{x^1}\hcA+\cdots+\partial_{x^{a-1}}\hcA$ (and the similar statements for $g_i$, $i=1,\dots,a-1$, obtained by relabeling of the independent variables) is equivalent to $\omega\in d_H \Omega^{a-2}$. That is, we have to prove that the auxiliary complex~\eqref{eq:auxiliary-complex} is acyclic in cohomological degree $a-1$. 

In order to do this, we revisit the argument of Anderson, cf.~\cite[Proposition 4.2 and 4.3]{anderson}. While the horizontal rows of the variational bicomplex that is considered in~\cite{anderson} are, in general, quite different from~\eqref{eq:auxiliary-complex} (the meaning of the symbol $\theta$ is completely different),
the combinatorics of the differential $d_H$ is literally the same, which allows us to adapt the formulas of Anderson. 

The first thing we need to do is to refine the standard gradation. Namely, we represent $S\in\mathbb{Z}^D_{\geq 0}$, $|S|>0$, as the sum $S=S'+S''$, where $S'=\sum_{i=1}^a s_i\xi_i$ and $S''=\sum_{i=a+1}^D s_i\xi_i$. This way we represent $\deg$ as the sum $\deg'+\deg''$, where $\deg'=|S'|$ and $\deg''=|S''|$. 
Observe that both $\deg''$ and the super gradation $\deg_\theta$ are preserved by the differential $d_H$. This means that we can split the complex~\eqref{eq:auxiliary-complex} into the direct summands 
\begin{equation}\label{eq:auxiliary-complex-summand}
0\to\left(\Omega^p_{d''}\right)^0\xrightarrow{d_H} \left(\Omega^p_{d''}\right)^1 \xrightarrow{d_H} \cdots \xrightarrow{d_H} \left(\Omega^p_{d''}\right)^{a-1}\xrightarrow{d_H} \left(\Omega^p_{d''}\right)^a,
\end{equation}
where the coefficients have $\deg''=d''$ and $\deg_\theta = p$. 

There are two different cases, $p=0$ and $p\not=0$. In the case $p=0$, we have no $\theta$s. Also, we consider the variables $u^{\alpha,S}$ as $u^{(\alpha,S''),S'}$, that is, we introduce a new index $(\alpha,S'')$ for the dependent variables, and we take into account only the dependence on the independent variables $x^1,\dots,x^a$. In other words, we redefine the algebra $\cA$ to be 
$$
C^\infty(U)[[u^{(\alpha,S'')}, \alpha=1,\dots,N, |S''|>0]] [[u^{(\alpha,S''),S'}, |S''|>0]], 
$$
where $\{u^{(\alpha,S'')}\}$ is the new set of dependent variables (we allow $S''=0$ and identify $u^{(\alpha,0)}$ with $u^\alpha$), $\{x^1,\dots,x^a\}$ is the new set of independent variables, and $u^{(\alpha,S''),S'}=\partial^{S'}u^{(\alpha,S'')}$.  If we fix the degree $\deg''$, we still have a finite number of dependent variables, and the only difference with the standard case is that we require our functions to be homogeneous polynomials in some of them (as opposed to just smooth functions). This modification (and also the minor difference that we don't consider explicit dependence on independent variables) is the only difference  between the complex~\eqref{eq:auxiliary-complex-summand} for $p=0$ and the complex considered in the first half of~\cite[Proposition 4.3]{anderson}. It is then straightforward to see that the argument of Anderson proves that this complex is acyclic (up to cohomological degree $(a-1)$). 

It remains to prove that the complex~\eqref{eq:auxiliary-complex-summand} is acyclic in cohomological degree $\leq (a-1)$ for $p>0$. In this case we can follow ~\cite[Proof of Proposition 4.2]{anderson}. Anderson gives an explicit homotopy contraction operator that uses only the combinatorics of indices of the variables $\d^{S'}\theta_\alpha^{S''}=:\theta_{(\alpha,S'')}^{S'}$, so we can also define it as an operator $h^{p,i}\colon \left(\Omega^p_{d''}\right)^i\to \left(\Omega^p_{d''}\right)^{i-1}$. Since the bookkeeping of indices and notation in~\cite{anderson} is quite different from the one we use in this paper, let us rewrite the homotopy contraction operator explicitly in our terms, cf.~\cite[Equation (4.13)]{anderson}:
\begin{align} \label{eq:HomotopyContraction}
& h^{p,i}:= \frac{-1}{p} \sum_{j=1}^a \sum_{(\alpha,S'')} 
\sum_{\substack{f_1,\dots,f_a\geq 0 \\ g_1,\dots,g_a\geq 0}} 
\frac{f_j+g_j+1}{a-i+1+\sum_{\ell=1}^a f_\ell}\cdot 
\prod_{\ell=1}^a \binom{f_\ell+g_\ell}{f_\ell} \cdot
\\ \notag
& \left[\prod_{\ell=1}^a \partial_{x^\ell}^{f_\ell}\right]\circ\theta^0_{(\alpha,S'')}\wedge
\left[\prod_{\ell=1}^a (-\partial_{x^\ell})^{g_\ell}\right]\circ
\frac{\partial}{\partial \theta_{(\alpha,S'')}^{\xi_j +\sum_{i=1}^a (f_i+g_i)\xi_i}}\circ
\frac{\partial}{\partial \left(dx^j\right)}
\end{align}
Then the arguments of~\cite[Lemma 4.4 and Proof of Proposition 4.2]{anderson} imply that $(h^{p,i+1} d_H + d_H h^{p,i}) \omega = \omega$ for $\omega\in \left(\Omega^p_{d''}\right)^i$, $p>0$ and $i<a$. This implies that for $p>0$ the cohomology of the chain complex~\eqref{eq:auxiliary-complex-summand} in cohomological degree $\leq (a-1)$ is equal to zero. 

This completes the proof of the Lemma.
\end{proof}

\subsection{Deformations of DN brackets and Poisson cohomology}
\label{27}

Let us introduce the transformations
\begin{equation}\label{eq:Miuradef}
 u^i\mapsto \tilde{u}^i=\sum_{k=0}^\infty\epsilon^k F^i_k(u;u_I),\;i=1,\ldots,N
\end{equation}
on the space $\cA$, where $F^i_k\in\cA_k$ and
\begin{equation*}
 \det\left(\frac{\d F^i_0(u)}{\d u^j}\right)\neq0.
\end{equation*}
The transformations \eqref{eq:Miuradef} form a group who is called the \emph{Miura group} \cite{dz01}. It can be regarded as the group of local diffeomorphisms on the space $\cA$, whose Lie algebra is the algebra of the (translational invariant) vector fields on $\cA$. The transformation of the 0-th order coordinates $u^i$ is then lifted to the higher order jet variables $u^{i,S}$. The action of the elements of the Miura group is then naturally extended to the full space $\hcA$. An important subclass of Miura transformations, that plays a central role in the theory of the deformations of DN brackets, are the so-called \emph{second kind Miura deformations} \cite{lz11}, for which $F^i_0=u^i$.

\begin{definition}\label{defDefo}
 Given a Poisson bivector $P_0\in\hcF^2$, a $n$-th order infinitesimal \emph{compatible deformation} of $P_0$ is a bivector $P=P_0+\sum_{k=1}^n\epsilon^kP_k$ such that $[P,P]=O(\epsilon^{n+1})$. The bracket associated to a deformed bivector by the rule $\{F,G\}=[[P,F],G]$ is then
 \begin{equation}\label{defbrack}
 \{,\}^{\sim}=\{,\}^0+\sum_{k=1}^n\epsilon^k\{,\}^k.
 \end{equation}
 In the DN case, where $\deg P_0=1$, the degree of each deformation $\deg P_k$ is $k+1$.
\end{definition}
\begin{definition}\label{defTriv}
 A deformation of $P_0$ is said to be trivial if there exists an element $\phi$ of the Miura group such that $\phi_*P=P_0$. From Definition \ref{defDefo}, this implies that $\phi$ must be of second kind. Equivalently, an infinitesimal deformation is trivial if there exist an evolutionary vector field $X$ such that $[X,P_0]=P$. This is equivalent to say that $\{\phi(u(x)),\phi(u(y))\}_0=\phi\left(\{u(x),u(y)\}^\sim\right)+O(\epsilon^{n+1})$ for a deformed bracket of degree $n+1$.
\end{definition}

As in the finite dimensional setting, the theory of deformations of DN brackets can be rephrased in terms of the \emph{Poisson--Lichnerowicz cohomology}. The main result of this paper is, indeed, the computation of the full cohomology for $N=1$ and arbitrary $D$ DN brackets. It is then well known the relation between lower order cohomology groups and the key elements of the deformation theory.

If $P\in\hcF^2$ and $[P,P]=0$, then $D_P$ defined in \eqref{eq:defDP} squares to 0. That means that it is possible to define a cochain complex
\begin{equation}
 0\rightarrow\hcA^0\xrightarrow{D_P}\hcA^1\xrightarrow{D_P}\hcA^2\xrightarrow{D_P}\cdots
\end{equation}
and its cohomology; moreover, since $D_P$ commutes with all the $\d_{x^i}$ the complex and the cohomology groups pass to the quotient space $\hcF$.

We denote
\begin{equation}
 H^p(\hcA)=\frac{\ker\left(D_P\colon\hcA^p\to\hcA^{p+1}\right)}{\im\left(D_P\colon\hcA^{p-1}\to\hcA^p\right)}
\end{equation}
and
\begin{equation}\label{defcohint}
 H^p(\hcF)=\frac{H^p(\hcA)}{\d_{x^i}\hcA+\cdots+\d_{x^D}\hcA}=\frac{\ker\left(d_P\colon\hcF^p\to\hcF^{p+1}\right)}{\im\left(d_P\colon\hcF^{p-1}\to\hcF^p\right)}
\end{equation}

The groups $H^\bullet(\hcF,d_P)$ constitute the Poisson--Lichnerowicz cohomology in the infinite dimensional setting we are dealing with. We identify the first cohomology group $H^1$  with the symmetries of the Poisson bivector $P$ that are not Hamiltonian, and the second cohomology group $H^2$ with the infinitesimal \emph{compatible deformations} of the Poisson bracket defined by the bivector $P$ that are not \emph{trivial}. Recalling the definition of $d_P$, a symmetry $X$ is an (evolutionary) vector field, namely a derivation of $\hcF$ which commutes with all $\{\partial_{x^i}\}$, such that $[P,X]=0$, a Hamiltonian vector field is a vector field of form $X_H=[P,H]$ for $H\in\hcF^0$ a local functional, a compatible bivector $P'$ is a bivector such that $[P,P']=0$, and a trivial compatible bivector is such that $P'=[P,Y]$ for some vector field $Y\in\hcF^1$.

The gradation on $\hcA$ defined in Paragraph 2.3 can be used to decompose the cohomology groups both on $\hcA$ and $\hcF$. We will denote, for instance, $H^2_2(\hcF)$ the first order ($n=1$ in the expansion \eqref{defbrack}) nontrivial compatible deformations of a Poisson bivector of degree 1.

\begin{remark}
 Spelling out the compatibility condition for the bivectors, one get a sequence of equations
 \begin{align}
  [P_0,P_1]&=0\\
  2[P_0,P_2]+[P_1,P_1]&=0\\
  \vdots&
 \end{align}
In particular, the first nonvanishing term of the expansion must be a compatible bivector in the Bihamiltonian sense, and the possibility of extending the deformation from the first order to the following ones is related to the third cohomology group.

Furthermore, if we have a bivector $P=\sum_{k=0}^n P_k$, $\deg P_k = k+1$, $[P,P]=O(\epsilon^{n+1})$, we can extend it one order higher in $\epsilon$ if and only if the following element of $\hcF^3$ happens to be $d_{P_0}$-exact:
\[
[P_1,P_{n}]+[P_2,P_{n-1}]+\cdots+[P_n,P_1].
\]
Note that the condition $[P,P]=O(\epsilon^{n+1})$ implies that this element is $d_{P_0}$-closed. To this end it would be sufficient to have $H^3_{n+3}(\hcF)=0$, but unfortunately, as we see below, it is in general not the case for the DN-brackets.

However, once we have a particular extension to order $\epsilon^{n+1}$ given by a bivector $P_{n+1}$ such that $[P+P_{n+1},P+P_{n+1}]=O(\epsilon^{k+2})$ (for instance, we definitely have one equal to zero in the case $P_1=\cdots=P_n=0$), then we can describe explicitly the space of all possible extensions up to Miura transformations of second kind. It is an affine space given by $P_{n+1}+H^2_{n+1}(\hat F)$. 
\end{remark}

\section{The Poisson cohomology in the scalar case}
\label{3}
In this Section we consider the $N=1$, or scalar, case. 

\subsection{Main result}
Our main result can be formulated as follows. Let $\Theta$ be the polynomial ring $\R[\{\theta^S, S\in (\Z_{\geq0})^{D-1}\}]$. The derivations $\partial_{x^i}$, $i=1, \dots , D-1$ act on $\Theta$ in the obvious way. 
Denote by $H^p_d(D)$ the homogeneous component of bi-degree $(p,d)$ of 
\begin{equation}
H(D) = \frac{\Theta}{\partial_{x^1} \Theta + \cdots + \partial_{x^{D-1}} \Theta } .
\end{equation}
\begin{theorem} \label{mainthm}
The Poisson cohomology of the Poisson bracket~\eqref{scalpb} in bi-degree $(p,d)$ is isomorphic to the sum of vector spaces
\begin{equation}
H^p_d(D) \oplus H^{p+1}_d (D) .
\end{equation}
\end{theorem}

The proof of this Theorem will be given in the following subsections. The strategy is to compute first the cohomology of a particular Poisson bracket, and then to show that the cohomology does not change under linear changes of the independent variables. That allows us to extend the result to the whole class of Poisson brackets~\eqref{scalpb}.

Let us first derive some consequences of Theorem~\ref{mainthm}. Let's start by an explicit description of the spaces $H^p_d(D)$ for small or large $p$:
\begin{lemma}\label{lemmaH(D)}
We have that
\begin{equation}
H^p_d(D) = \begin{cases}
\R & d=0, \ p=0,1,\\
0 & d\geq1,\ p=0,1,\\
0 & d=2, \ p=2, \\
\R^{\binom{D-1}{d}} &d\geq0, \ p=d+1, \\
0	& d\geq0, \ p\geq d+2.
\end{cases}
\end{equation}
\end{lemma}

\begin{remark}
The vanishing result in last line of the Lemma can be improved:  $H^p_d(D)$ is zero for any $(p,d)$ in the range
\begin{equation} \label{vanibi}
\binom{D+l-1}{l}< p \leq \binom{D+l}{l+1}, \quad
0 \leq d <p(l+1) - \binom{D+l}{l}
\end{equation}
for any $l\geq0$.
\end{remark}

\begin{proof}
For $p=0,1$ the statement is trivial. 

For $p\geq d+2$, it is easy to see that $\Theta_d^p = 0$, therefore $H^p_d(D) =0$ as well. By computing the minimal $\deg$ of a monomial of $\deg_\theta = p$, one can also extend this vanishing result for $\Theta^p_d$ to the range given in~\eqref{vanibi}.

For $p=d+1$ a basis of $\Theta_d^{d+1}$ is given by
\begin{equation}
\theta^0 \theta^{\xi_{i_1}} \cdots \theta^{\xi_{i_d}}
\end{equation}
with $i_1 < \cdots < i_d$ and $i_l = 1, \dots , D-1$. It follows that
\begin{equation}
\dim H^{d+1}_d (D) =\dim \Theta^{d+1}_d = \binom{D-1}{d}.
\end{equation}

For $p=d=2$ a generic element in $\Theta^2_2$ 
\begin{equation}
\sum_{i,j=1}^{D-1} \left( a_{ij} \theta^{\xi_i} \theta^{\xi_j} + 
b_{ij} \theta \theta^{\xi_i +\xi_j} \right)
\end{equation}
can always be cancelled by a generic element in $\sum_{i=1}^{D-1} \partial_{x^i} \Theta^2_1$, which is of the form 
\begin{equation}
\sum_{i=1}^{D-1} \partial_{x^i} \left( \sum_{j=1}^{D-1} c_{ij} \theta \theta^{\xi_j}\right) =
\sum_{i,j=1}^{D-1} c_{ij} \left( \theta^{\xi_i} \theta^{\xi_j} + 
 \theta \theta^{\xi_i +\xi_j} \right),
\end{equation}
by the obvious symmetry properties of the matrices $a_{ij}$, $b_{ij}$, $c_{ij}$.
\end{proof}

\begin{remark}
For $D\geq2$ the dimension of $\Theta^p_d$ is given by
\begin{equation}
\dim \Theta_d^p = \sum_{\substack{p_0, p_1, \dots  \geq 0 \\ p_0 + p_1 + \cdots = p \\ p_1 + 2 p_2 + 3p_3 + \cdots = d}}  
\prod_{s\geq0} \binom{\binom{s+D-2}{D-2}}{p_s} 
\end{equation}
which can be given in terms of a generating function as
\begin{equation}
\sum_{p,d\geq 0} \dim \Theta_d^p \ x^p y^d = \prod_{s\geq 0} (1+ x y^s)^{\binom{s+D-2}{D-2}} .
\end{equation}
\end{remark}

A straightforward application of Theorem~\ref{mainthm} gives us the explicit form of the Poisson cohomology groups in several cases:
\begin{corollary} \label{explgr}
We have that:
%
\begin{equation}
H^p_d(\hcF) \cong \begin{cases}
\R^2 & d=0, \ p=0, \\
\R & d=0, \ p=1, \\
0 & d\geq1, \ p=0, \\
\R^{D-1} & d=1, \ p=1, \\
0 & d=1, \ p=2, \\
\R^{\frac12 (D-1)(D-2)} &d=2, \ p=2, \\
\R^{\binom{D-1}{d}} & d\geq0 , \ p=d+1, \\
0 &d\geq0, \ p\geq d+2 .
\end{cases}
\end{equation}
\end{corollary}

\begin{remark}
The vanishing result in the last line actually generalises to the set of bidegrees~\eqref{vanibi}.
\end{remark}

\begin{remark}
In the $D=1$ case we have $H(1) = \Theta = \R \oplus \R \theta^0$, therefore Theorem~\ref{mainthm} implies that $H^p_d(\hcF)$ is isomorphic to $\R^2$ for $p=d=0$, to $\R$ for $p=1$, $d=0$, and vanishes otherwise, as in the scalar case of Getzler's result~\cite{g02}. 
\end{remark}

\begin{remark}
For $D=2$ we can provide a general formula for $H^p_d(\hcF)$. Indeed, in this case $\Theta=\R[\theta^{(i,0)}]$ and in particular $\Theta^p_d$ is generated by the monomials $\theta^{(i_1,0)}\theta^{(i_2,0)}\cdots\theta^{(i_p,0)}$ such that $i_1+i_2+\cdots+i_p=d$, $i_k\geq 0$. The dimension of $\Theta^p_d$ is given by the number of ways of writing $d$ as the sum of $p$ distinct nonnegative integers, regardless of the order. The result, in terms of the partition function $P(n,k)$ giving the number of ways of writing $n$ as sum of $k$ positive addends, is \cite{c74}
\begin{equation}\label{dimtheta}
 \dim \Theta^p_d=P\left(d+p-\binom{p}{2},p\right).
\end{equation}
Since $\ker\partial_{x^1}=\R\neq 0$ only in $\Theta^0_0$, the dimension of $\partial_{x^1}\Theta^p_{d-1}$ for $(p,d)\neq(0,1)$ is given by the same formula \eqref{dimtheta} with $d'=d-1$. 
%
%
 
Then we have:
\begin{equation}
\dim H^p_d(2) = P(d+\frac{p(3-p)}2 , p) -  P(d-1+\frac{p(3-p)}2 , p) 
\end{equation}
for $(p,d)\not=(0,1)$ and $\dim H^0_1(2) = 0$.
The dimension of the cohomology group $H^p_d(\hcF)$ is obtained in a straightforward way from Theorem \ref{mainthm} for $D=2$.
We can conveniently express it as a generating function
\begin{equation} 
\sum_{d\geq0} \dim H^p_d(\hcF) \ x^d = h^p(x) + h^{p+1}(x) ,
\end{equation}
where
\begin{equation}
\sum_{d\geq0} \dim H^p_d(D=2) \ x^d = h^p(x) := x^{\frac{p}2 (p-1)} \prod_{i=2}^p (1-x^i)^{-1}  +x \delta_{p,0}.
\end{equation}
From the explicit form of $P(n,k)$ for $k=2,3$ we can improve the  results of Lemma \ref{lemmaH(D)}
\begin{equation}
\dim H^2_d(D=2)=\begin{cases}
             0 & d=2k,\\
	     1 & d=2k+1,
            \end{cases}
\end{equation}
\begin{equation}
 \dim H^3_d(D=2)=\begin{cases}
                  0 & d<3 ,\\
		k & d=4+6k ,\\
		k+1 & d=3+6k, \ 5+6k\leq d \leq 8+6k.
\end{cases}
 \end{equation}

Combining the two previous results we obtain:
\begin{equation} \label{HD2}
\dim H^2_d(\hcF) = \begin{cases}
k & d=4+6k, \\
k+1 & d=6+6k, \ d=8+6k , \\
k+2 & d=3+6k, \ d=5+6k, \ d=7+6k .
\end{cases}
\end{equation}
In particular it should be noticed that $H^2_d(\hcF)\neq 0$ for all $d>4$.
Some explicit values are
\begin{center}
\begin{tabular}{|r | *{9}{c}|}
\hline
$d$ & 0 & 1 & 2 & 3 & 4 & 5 & 6 & 7 & 8\\
\hline
$\dim H^2_d(\hcF)$ & 0 & 1 & 0 & 2 & 0 & 2 & 1 & 2 & 1\\
\hline
$\dim H^3_d(\hcF)$ & 0 & 0 & 0 & 1 & 0 & 1 & 2 & 1 & 2\\
\hline
\end{tabular}
\end{center}
\end{remark}

\begin{remark}
For $D>2$ we have no explicit formulas, but we expect $H^p_d(D)$ to be generically non-trivial. Symbolic computer calculations allow us to obtain, for example, for $D=3$:
\begin{center}
\begin{tabular}{|r | *{9}{c}|}
\hline
$d$ & 0 & 1 & 2 & 3 & 4 & 5 & 6 & 7 & 8\\
\hline
$\dim H^2_d(\hcF)$ & 0 & 2 & 1 & 8 & 3 & 16 & 13 & 26 & 26\\
\hline
$\dim H^3_d(\hcF)$ & 0 & 0 & 1 & 4 & 6 & 14 & 29 & 36 & 72\\
\hline
\end{tabular}
\end{center}
and for $D=4$:
\begin{center}
\begin{tabular}{|r | *{7}{c}|}
\hline
$d$ & 0 & 1 & 2 & 3 & 4 & 5 & 6 \\
\hline
$\dim H^2_d(\hcF)$ & 0 & 3 & 3 & 20 & 15 & 66 & 73 \\
\hline
$\dim H^3_d(\hcF)$ & 0 & 0 & 3 & 11 & 30 & 75 & 183 \\
\hline
\end{tabular}
\end{center}

\end{remark}

\subsection{Cohomology in a special case}\label{cohosimple}

Let us consider the DN brackets with one dependent variable and $D$ independent variables
\begin{equation}
\{ u(x) , u(y) \}^{\hat{}} = \partial_{x^D} \delta(x-y),
\end{equation}
which, in the $\theta$ formalism, corresponds to the bivector
\begin{equation}
\hat{P} = \frac12 \int^D \theta \theta^{\xi_D} \ d^D x \in \hcF^2. 
\end{equation}
In this Section we compute the Poisson cohomology of the bracket $\{,\}^{\hat{}}$, i.e., the cohomology of the complex $(\hcF, d_{\hat{P}})$, where $d_{\hat{P}}= [\hat{P},\cdot]$.

Let
\begin{equation} \label{deltaop}
\Delta := D_{\hat{P}} = \sum_S \theta^{S+\xi_D} \frac{\partial }{\partial u^S} ,
\end{equation}
be the differential operator on $\hcA$ associated with $\hat{P}$ defined in~\eqref{eq:defDP}. 
Clearly, $\Delta$ squares to zero. 
Moreover, this operator commutes with $\d_{x^1},\dots,\d_{x^D}$, therefore it induces a differential on each of the spaces $\hcF_1,\dots,\hcF_D$. In particular, it is obvious that on $\hcF_D = \hcF$ this operator coincides with $d_{\hat{P}}$.

With the differentials induced by $\Delta$ on $\hcA,\hcF_1,\dots,\hcF_D$, the short exact sequences~\eqref{eq:Seq} become short exact sequences of complexes. We will use the associated long exact sequences in order to compute the cohomology of $\hcF_1,\dots,\hcF_D$.
Since all the cohomology groups that we consider are to be understood with respect to the differential induced by $\Delta$, we will generally refrain from indicating it all the time.

Recall that we denote by $Z_D$ the sub-semiring of $Z$ that consists of multi-indices of the form $S=\sum_i s_i \xi_i$ with $s_D =0$. 
\begin{lemma} \label{lemmaH}
We have that $H(\hcA) = \R[\{ \theta^S, S\in Z_D \} ]$.
\end{lemma}
\begin{proof}
Under the identification $\theta^{S+\xi_D}\leftrightarrow du^S$, $S\in Z$, the operator $\Delta$ turns into the de Rham differential on the differential forms in $u^S$, $S\in Z$, hence its cohomology is given by the degree $0$ forms constant in $u^S$, $S\in Z$, that is, by all elements of $\hcA$ independent of $u^S$ and $\theta^{S+\xi_D}$, $S\in Z$. These are the 
polynomials in $\theta^S$, $S\in Z_D$.
\end{proof}

For $D=1$ we have $H(\hcA)= \R \oplus \R \theta$, as in~\cite{lz11, lz13}.

As before, we denote the polynomial ring $\R[\{\theta^S, S\in Z_D\}]$ by $\Theta$. It is important to stress that the polynomials in $\theta^S$, $S\in Z_D$, are indeed representatives of cohomology classes in $H(\hcA)$.

\begin{proposition}
We have:
\begin{equation}
H(\hcF_i) = \frac{\Theta}{\partial_{x^1} \Theta + \cdots + \partial_{x^i} \Theta} 
\end{equation}
for $i=1, \dots , D-1$.
\end{proposition}
\begin{proof}
Let us start by considering the short exact sequence 
\begin{equation}
0\to\hcA/\R\xrightarrow{\d_{x^1}}\hcA\xrightarrow{\int dx^1}\hcF_1\to 0 ,
\end{equation}
which induces a long exact sequence in cohomology
\begin{equation}
H^p_{d-1}(\hcA/\R) \xrightarrow{\d_{x^1}} H^{p}_d(\hcA) \xrightarrow{\int d x^1} H^p_d(\hcF_1)\to 
H^{p+1}_{d}(\hcA/\R) \xrightarrow{\d_{x^1}} H^{p+1}_{d+1}(\hcA).
\end{equation}
By Lemma~\ref{lemmaH}, the cohomology classes in $H(\hcA)$, and of course in $H(\hcA/\R)$, are represented by elements of $\Theta$, i.e., by polynomials in $\theta^S$ with $S\in Z_D$. The derivations $\partial_{x^1}, \dots , \partial_{x^{D-1}}$ map such polynomials to polynomials of the same form, hence their action on the cohomology coincides with their natural action on $\Theta$. It clearly follows that the kernel of $H^{p+1}_{d}(\hcA/\R) \xrightarrow{\d_{x^1}} H^{p+1}_{d+1}(\hcA)$ is equal to zero. 

The Bockstein homomorphism 
$H^p_d(\hcF_1)\to H^{p+1}_{d}(\hcA/\R)$ 
is therefore the zero map, and we can conclude that $H^p_d(\hcF_1)$ is the quotient of $H^{p}_d(\hcA)$ by the image of $\partial_{x^1}$. So, we have
\begin{equation}
H^p_d(\hcF_1) \cong \frac{H^{p}_d(\hcA)}{\d_{x^1}H^p_{d-1}(\hcA)}\cong \frac{\Theta^p_d}{\d_{x^1}\Theta^p_{d-1}},
\end{equation}
which is precisely the assertion of the Proposition for $\hcF_1$. 

Let us now prove the same statement for $\hcF_i$, $i=2,\dots,D-1$, by induction. The short exact sequence
\begin{equation}
0\to\hcF_{i-1}/\R\xrightarrow{\d_{x^i}}\hcF_{i-1}\xrightarrow{\int dx^i}\hcF_i\to 0
\end{equation}
induces the long exact sequence in cohomology
\begin{equation}
H^p_{d-1}(\hcF_{i-1}/\R) \xrightarrow{\d_{x^i}} H^{p}_d(\hcF_{i-1}) \xrightarrow{\int dx^i} H^p_d(\hcF_i)\to 
H^{p+1}_{d}(\hcF_{i-1}/\R) \xrightarrow{\d_{x^i}} H^{p+1}_{d+1}(\hcF_{i-1}).
\end{equation}
Notice that the map $\d_{x^i}\colon H(\hcF_{i-1}/\R)\to H(\hcF_{i-1})$ is given, by inductive assumption, by
\begin{equation}\label{eq:map-d1-theta}
\d_{x^i} \colon \frac{\Theta}{\d_{x^1} \Theta +\cdots + \d_{x^{i-1}} \Theta + \R} \to \frac{\Theta}{\d_{x^1} \Theta +\cdots + \d_{x^{i-1}} \Theta} .
\end{equation}

\begin{lemma}
	The kernel of the map~\eqref{eq:map-d1-theta} is equal to zero.
\end{lemma}

\begin{proof} We can follow the proof of Lemma~\ref{lemma:kernel-di-zero}. By the same argument as there, we reduce the statement of the lemma to the vanishing of the cohomology of a certain complex, where we have an explicit homotopy contraction given by exactly the same formula as in Equation~\eqref{eq:HomotopyContraction}. 
\end{proof}

Since the kernel of the map~\eqref{eq:map-d1-theta}  is equal to zero, the Bockstein homomorphism vanishes, and therefore we can conclude that
\begin{equation}
H(\hcF_i)\cong \frac{\frac{\Theta}{\d_{x^1} \Theta +\cdots + \d_{x^{i-1}} \Theta}}{\d_{x^i}\left(\frac{\Theta}{\d_{x^1} \Theta +\cdots + \d_{x^{i-1}} \Theta}\right)}
\cong
\frac{\Theta}{\d_{x^1} \Theta +\cdots + \d_{x^i} \Theta} .
\end{equation}
\end{proof}

In particular, the previous Proposition implies that $H^p_d(\hcF_{D-1})$ is a finite dimensional vector space, whose dimension one can compute explicitly for each choice of $p,d,D$. Due to its importance we denote it by $H^p_d(D)$, which is then the degree $(p,d)$ component of 
\begin{equation}
\frac{\Theta}{\partial_{x^1} \Theta + \cdots + \partial_{x^{D-1}} \Theta}.
\end{equation}

Finally we can compute the Poisson cohomology in the scalar case:
\begin{theorem}
We have that $H^p_d(\hcF)\simeq H^p_d(D) \oplus H^{p+1}_{d}(D)$.
\end{theorem}
\begin{proof}
The short exact sequence
\begin{equation}
0\to\hcF_{D-1}/\R\xrightarrow{\d_{x^D}}\hcF_{D-1}
\xrightarrow{\int dx^D}\hcF_D\to 0
\end{equation}
induces in cohomology the exact sequence
\begin{equation}\label{eq:LES-D}
H^p_{d-1}(\hcF_{D-1}/\R) \xrightarrow{\d_{x^D}} H^{p}_d(\hcF_{D-1}) \xrightarrow{\int dx^D}  H^p_d(\hcF_D)\to
H^{p+1}_{d}(\hcF_{D-1}/\R) \xrightarrow{\d_{x^D}} H^{p+1}_{d+1}(\hcF_{D-1}).
\end{equation}

It is easy to check that, for $a\in\Theta$,
\begin{equation}
\partial_{x^D} a = \Delta \sum_S u^S \frac{\partial a}{\partial \theta^S}, 
\end{equation}
hence the map $\d_{x^D}$ sends any element of $\Theta$, considered as a subspace of $\hcA$, to a $\Delta$-exact element of $\hcA$. 
This implies that the operator induces by $\d_{x^D}$ on the cohomology of $\hcA$ is equal to zero, and therefore it is also equal to zero on the cohomology of $\hcF_i$, $i=1,\dots,D-1$. The long exact sequence~\eqref{eq:LES-D} is then equivalent to the collection of short exact sequences
\begin{equation}
0\to H^{p}_d(\hcF_{D-1}) \to H^p_d(\hcF_D)\to 
H^{p+1}_{d}(\hcF_{D-1}/\R) \to 0,
\end{equation}
for $p\geq 0$, $d\geq 0$. This implies that $H^p_d(\hcF_D)\simeq H^{p}_d(\hcF_{D-1}) \oplus 
H^{p+1}_{d}(\hcF_{D-1})$.
\end{proof}


\subsection{Change of independent variables}

In Section \ref{cohosimple} we have proved a theorem about the Poisson cohomology for the bracket $\{u(x),u(y)\}=\partial_{x^D}\delta(x-y)$. On the other hand, the generic nondegenerate Poisson bracket~\eqref{scalpb} has the form, when written in flat coordinates, 
\begin{equation} \label{flatpb}
\{ u(x) , u(y) \} = \sum_{i=1}^D c^i \frac{\partial }{\partial x^i} \delta(x-y) .
\end{equation}
Let us denote $\Delta$ the Poisson bivector associated to the bracket \eqref{flatpb}.

In this Section we prove that the cohomology groups for the bracket \eqref{flatpb} are isomorphic to the ones we computed in the previous one, and hence that Theorem \ref{mainthm} holds for all the nondegenerate scalar DN brackets.

\begin{lemma}By a linear change of independent variables $(x^1,\ldots,x^D)$ the bracket \eqref{flatpb} can be brought to the form
\begin{equation}\label{eq:Normal3}
 \{v(\tilde{x}),v(\tilde{y})\}=\d_{\tilde{x}^D}\delta(\tilde{x}-\tilde{y}).
\end{equation}
\end{lemma}
\begin{proof}
Under a linear change of coordinates $x\mapsto \tilde{x}=J x$, with $J$ a constant $D\times D$ matrix, the derivations $\partial_x=\{\partial_{x^i}\}$ change according to $\d_{x}\mapsto\d_{\tilde{x}}=(J^{-1})^T\d_x$. For the $D$-dimensional Dirac's delta function to be invariant, we require in addition that $\det J=1$. Introducing the $D$-dimensional vector $C=(c_1,\ldots,c_D)$ we have to solve the algebraic set of equations $J\cdot C=\xi_D$. We get $D$ equations for $D^2-1$ entries of the matrix $J$. Hence, the solution is not unique but it always exists.
\end{proof}
We will denote $\tilde{\Delta}$ the Poisson bivector defining the bracket \eqref{eq:Normal3}.
\begin{lemma} The cohomology groups $H^p_d(\hcF,\Delta)$ are isomorphic to $H^p_d(\hcF,\tilde{\Delta})$.
\end{lemma}
\begin{proof}
The same transformation law $\d_{x}\mapsto\d_{\tilde{x}}=(J^{-1})^T\d_x$ applies to the variables $u^{\xi_i}=\d_{x^i}u$ and $\theta^{\xi_i}$. Since $J$ is constant, the higher order derivatives leave it unaffected, in such a way that the transformation law is tensorial
\begin{equation}
 u^I\mapsto \tilde{u}^I=\left((J^{-1})^T\right)^{|I|}u^{I'}
\end{equation}
where $|I|=k_1+\ldots+k_d$ and $|I|=|I'|$. More precisely, we have
\begin{multline}
 \tilde{u}^{(k_1,\ldots,k_D)}=\\
\left((J^{-1})^T\right)^{i_1}_{1}\cdots\left((J^{-1})^T\right)^{i_{k_1}}_{1}\left((J^{-1})^T\right)^{i_{k_1+1}}_{2}\ldots\left((J^{-1})^T\right)^{i_{|I|}}_{D}\frac{\d^{|I|}}{\d x^{i_1}\ldots\d x^{i_{|I|}}}u.
\end{multline}
Moreover, the partial derivatives with respect to the jets transform with $J^T$, so that the differential $\Delta$ is transformed, being homogeneous of differential order $1$, as $(J^{-1})^T\Delta$. Under the change of independent variables, $\hcA^p_d\mapsto\hcA^p_d$ for any component of bidegree $(p,d)$.

Since $J$ is invertible and the change of coordinates does not change the differential order of the jets, there exists an isomorphism between the kernels and the images of $\Delta\colon\hcA^p_d\to\hcA^p_{d+1}$ and $\tilde{\Delta}\colon\hcA^p_d\to\hcA^p_{d+1}$. Hence the quotient spaces $H^p_d(\Delta, \hcA)$ and $H^p_d(\tilde{\Delta}, \hcA)$ are isomorphic.

Moreover, the change of independent variable leaves the space $\hcF$ invariant. Let us consider the quotient operation
$$\hcF=\frac{\hcA}{\d_{x^1}\hcA+\cdots+\d_{x^D}\hcA}.$$
The change of independent coordinates maps each partial derivative to a linear combination of all the $D$ derivatives
$$\d_{x^i}\mapsto\left((J^{-1})^T\d_{x}\right)_i=\sum_j\frac{\d x^j}{\d \tilde{x}^i}\d_{x^j}.$$
Since $J$ is nondegenerate, however, $\Span(\d_{\tilde{x}}\hcA)$ is the same as $\Span(\d_{x}\hcA)$, which proves the claim.

Equation \eqref{defcohint} reads
	$$
	H^p_d(\hcF)\cong \frac{\Theta^p_d}{\d_{x^1} \Theta^p_{d-1} +\cdots + \d_{x^D}\Theta^p_{d-1}}
	$$
where $\Theta^p_d=H^p_d(\hcA)$. The linear change of independent coordinates, as already discussed, maps the numerator of the quotient to an isomorphic space. For the denominator we have
\begin{multline*}
\d_{x^1}\Theta^p_{d-1}+\cdots\d_{x^D}\Theta^p_{d-1}\mapsto\\\d_{\tilde{x}^1}\tilde{\Theta}^p_{d-1}+\cdots\d_{\tilde{x}^D}\tilde{\Theta}^p_{d-1}=\d_{x^1}\tilde{\Theta}^p_{d-1}+\cdots\d_{x^D}\tilde{\Theta}^p_{d-1}\cong\d_{x^1}\Theta^p_{d-1}+\cdots\d_{x^D}\Theta^p_{d-1}.
\end{multline*}
Since both the numerator and the denominator of the quotient $H^p_d(\hcF)$ are mapped to isomorphic spaces, such is the quotient itself. This allows us to extend the Theorem \ref{mainthm} to all the nondegenerate scalar DN brackets.
\end{proof}

\section{Direct computation of some cohomology groups}

In this section we show how one can explicitly compute some cohomology groups $H^p_d(\hcF,\Delta)$ using the formalism of Poisson Vertex Algebras \cite{bdsk09}. The notion of Poisson Vertex Algebra was introduced as an algebraic framework to deal with evolutionary Hamiltonian PDEs with one spatial variable, but an extension to $D$-dimensional equations is straightforward and has already been exploited to compute some cohomology groups for $(D=2,N=2)$ DN brackets \cite{c14}.
\begin{definition}
A ($D$-dimensional) Poisson Vertex Algebra (PVA) is a differential algebra $(\cV,\{\d_i\}_{i=1}^D)$ endowed with  $D$ commuting derivations and with a bilinear operation $\{\cdot_{\mlambda}\cdot\}\colon \cV\otimes \cV\to \R[\lambda_1,\ldots,\lambda_D]\otimes \cV$ called the $\lambda$ \emph{bracket} satisfying the following set of properties:
\begin{enumerate}
\item $\{\d_i f_{\mlambda} g\}=-\lambda_i\{f_{\mlambda} g\}$
\item $\{f_{\mlambda}\d_i g\}=\left(\d_\alpha+\lambda_i\right)\{f_{\mlambda} g\}$
\item $\{f_{\mlambda} gh\}=\{f_{\mlambda} g\}h+\{f_{\mlambda} h\}g$
\item $\{fg_{\mlambda} h\}=\{f_{\mlambda+\md}h\}g+\{g_{\mlambda+\md}h\}f$
\item $\{g_{\mlambda}f\}=-{}_{\to}\{f_{-\mlambda-\md}g\}$ (PVA-skewsymmetry)
\item $ \{f_{\mlambda}\{g_{\mmu}h\}\}-\{g_{\mmu}\{f_{\mlambda}h\}\}=\{\{f_{\mlambda}g\}_{\mlambda+\mmu}h\}$ (PVA-Jacobi identity).
\end{enumerate}
We use a multi-index notation  $\mlambda^I=\lambda_1^{i_1}\lambda_2^{i_2}\cdots\lambda_D^{i_D}$ for $I=(i_1,i_2,\ldots,i_D)$. The terms in the RHS of Property (4) are to be read, if $\{f_{\mlambda}h\}=\sum_I B(f,h)_I \mlambda^I$, as $\{f_{\mlambda+\md}h\}g=\sum_IB(f,h)_I(\mlambda+\md)^Ig=\sum_IB(f,h)_I(\lambda_1+\d_1)^{i_1}\cdots(\lambda_D+\d_D)^{i_D}g$. Similarly, the skewsymmetry property \emph{in extenso} is $\sum_I B(g,f)_I\mlambda^I=-\sum_I(-\mlambda-\md)^I B(f,g)_I$.
\end{definition}
When we consider $\cV=\cA$ the algebra of differential polynomials and identify the ``total derivations'' $\d_{x^i}$ with $\d_i$, the set of axioms for the PVA translates into a practical formula that gives the bracket between two elements of $\cA$ in terms of the bracket between the so-called generators $u^\alpha$:
\begin{equation}\label{eq:master}
  \{f_{\mlambda}g\}=\sum_{\substack{\alpha,\beta=1\ldots,N\\L,M\in\Z^D_{\geq 0}}}\frac{\d g}{\d u^\beta_M}(\mlambda+\md)^M\{u^\alpha_{\mlambda+\md}u^\beta\}(-\mlambda-\md)^L\frac{\d f}{\d u^\alpha_L}.
\end{equation} 

The main reason why the notion of a Poisson Vertex Algebra is relevant and useful for the study of the Poisson cohomology is that there exists an isomorphism between the Poisson Vertex Algebras and the Poisson brackets on the space of local functionals. In particular, given the $\lambda$ bracket of a PVA we have that
\begin{equation}
 \left\{\int f,\int g\right\}=\int\{f_{\mlambda}g\}|_{\mlambda=0}
\end{equation}
and, given a Poisson bracket in the space of the local densities $\cA$, $\{u^\alpha(x),u^\beta(y)\}=\sum_I B^{\alpha\beta}_I(u(x);u_L(x))\md^I\delta(x-y)$, we can define a $\lambda$ bracket
\begin{equation}
 \{u^\alpha_{\mlambda}u^\beta\}=\sum_I B^{\beta\alpha}_S(u;u_L)\mlambda^S.
\end{equation}
Moreover, an evolutionary Hamiltonian PDE of form $u_t=\{\int h,u\}$ is mapped to $u_t=\{h_{\mlambda}u\}|_{\mlambda=0}$. The advantage of working with PVAs rather than on local functionals is that there is no need to integrate --- or, equivalently, to perform the quotient from $\cA$ to $\cF$ --- since the properties (5) and (6) encode the skewsymmetry and the Jacobi identity for the bracket after the integration. This allows to perform explicitly computations, a few examples of which we will now demonstrate.
\subsection{Symmetries of the bracket}
The $\lambda$ bracket is equivalent to the Poisson bracket defined by the bivector $\Delta=\frac{1}{2}\theta\theta^{\xi_D}$ is $\{u_{\mlambda}u\}=\lambda_D$. As we have already discussed in Sect. 2.4, the elements of the first cohomology group $H^1(\hcF,\Delta)$ are the symmetries of the bracket that are not Hamiltonian, namely the evolutionary vector fields
\begin{equation}
 X=\sum_I\md^I X(u;u_L)\frac{\d}{\d u_I}
\end{equation}
satisfying
\begin{equation}\label{eq:sym}
 X(\{u_{\mlambda} u\})=\{X(u)_{\mlambda}u\}+\{u_{\mlambda}X(u)\}
\end{equation}
of form different from
\begin{equation}
 X(u)=\{h_{\mlambda}u\}|_{\mlambda=0}
\end{equation}
Let us compute $H^1_0$ and $H^1_1$ and $H^1_2$ for generic $D$, showing the agreement with the results of Corollary \ref{explgr}.

For $H^1_0$ we want to consider an evolutionary vector field whose component $X$ depends only on $u$. Since the bracket between the generators $u$ is constant, the LHS of \eqref{eq:sym} vanishes; computing the RHS with the help of formula \eqref{eq:master} and setting it equal to 0 immediately gives that $X$ must be a constant. Given the form of the Poisson bivector, we cannot have any constant Hamiltonian vector fields; thus, $H^1_0=\R$.

To compute $H^1_1$ we are interested in vector fields of first degree, namely of form $X=\sum_{i=1}^DX^i(u)u_{\xi_i}$. Imposing the condition of symmetry does not give any condition on $X^D$ but forces $X^i$ for $i=1,\ldots,D-1$ to be constants. On the other hand, if we take a generic Hamiltonian density of degree 0, i.e., a function $h(u)$, its Hamiltonian vector field will be $X_h=h'(u)u_{\xi_D}$, where $X^D=h'$ can be arbitrary. This means that the $D-1$ constants only are elements of $H^1_1=\R^{D-1}$.

In principle, computing any component of the first cohomology group means performing the same computations as the ones for the first two. However, they become more and more involved with the growing of the degree (and hence of the differential order). Another advantage of the formulation of the problem in term of PVAs is that the computations are straightforward and can be performed by a computer. Computing \eqref{eq:sym} for a  degree two evolutionary vector field, whose component is $X=\sum_{a,b=1}^DX_1^{ab}u_{\xi_a}u_{\xi_b}+X_2^{ab}u_{\xi_a+\xi_b}$, we get that the necessary condition is that $X_1=X_2=0$ for all $(a,b)$. This means that $H^1_2=0$.

\subsection{Deformations of the bracket}
As already discussed, the second cohomology group classifies the compatible infinitesimal deformations of the $\lambda$ bracket associated with the Poisson bivector $\Delta$ that cannot be obtained by a Miura transformation. We will demonstrate a few results for the case $D=2$. Definition \ref{defDefo}, in terms of $\lambda$ bracket, translates into considering a deformed bracket
\begin{equation}
 \{\cdot_{\mlambda}\cdot\}=\{\cdot_{\mlambda}\cdot\}_\Delta+\epsilon\{\cdot_{\mlambda}\cdot\}^\sim
\end{equation}
and requiring that it satisfies Property (5) and (6) of the PVAs up to the order $\epsilon$. Since $\{\cdot_{\mlambda}\cdot\}_\Delta$ is constant, it is enough to impose the skewsymmetry of $\{\cdot_{\mlambda}\cdot\}^\sim$ and 
\begin{equation}\label{eq:defo}
 \{u_{\mlambda}\{u_{\mmu}u\}^\sim\}_\Delta+\{u_{\mmu}\{u_{\mlambda}u\}^\sim\}_\Delta=\{\{u_{\mlambda}u\}^\sim_{\mlambda+\mmu}u\}_\Delta.
\end{equation}
The form of the deformed bracket we choose depends on which component of the second cohomology group we are interested in: when computing $H^2_d$ we will consider homogeneous $\lambda$ brackets of degree $d$ (as for the gradation on $\hcA$, we consider $\deg\mlambda^I=|I|$ and $\deg\{u_{\mlambda}u\}=\deg(B(u;u_L)_I\mlambda^I)=\deg B+|I|$). Property (5) prevents the existence of nontrivial elements in $H^2_0$, since there cannot be skewsymmetric brackets of form $\{u_{\mlambda}u\}=A(u)$.

For $H^2_1$ we consider brackets of form $\{u_{\mlambda}u\}^\sim=\sum_{a=1}^22A^a(u)\lambda_a+A^{a}{}'u_{\xi_a}$ --- we have already implemented the skewsymmetry of the deformation. Computing  condition \eqref{eq:defo} for this bracket does not give any constraint on $A^2$ but implies that $A^1$ is constant. Let us consider a generic Miura transformation that gives rise to a bracket of  degree one: since $\Delta$ is  degree one as well, the transformation must be of degree 0, so $u\mapsto U= u+\epsilon F(u)$. We get
\begin{equation}
 \{U_{\mlambda} U\}_\Delta=\{u_{\mlambda} u\}_\Delta+\epsilon\left(2F'(u)\lambda_2+F''(u)u_{\xi_2}\right)+O(\epsilon^2).
\end{equation}
This computation shows that we can never get from a Miura transformation a bracket of form $\{u_{\mlambda}u\}=\lambda_1$, which is compatible with $\{\cdot_{\mlambda}\cdot\}_\Delta$. This means that $H^2_1=\R$.

We can proceed similarly for $H^2_2$; in this case a general deformed bracket of degree 2 will be
\begin{equation}
 \{u_{\mlambda}u\}^\sim=\sum_{a,b}A^{ab}\lambda_a\lambda_b+B^{ab}\lambda_au_{\xi_b}+C^{ab}u_{\xi_a}u_{\xi_b}+D^{ab}u_{\xi_a+\xi_b}.
\end{equation}
Note that, by definition, $A$, $C$, and $D$ are symmetric in the indices $(a,b)$. Imposing  skewsymmetry we find the relations
\begin{align}
 A^{ab}&=0,\\
C^{ab}&=\frac{1}{4}\left(B^{ab}{}'+B^{ba}{}'\right),\\
D^{ab}&=\frac{1}{4}\left(B^{ab}+B^{ba}\right).
\end{align}
We now impose  condition \eqref{eq:defo} on the bracket and find that the parameters $B$ must satisfy the set of equations
\begin{align}
 B^{11}&=0,\\
B^{12}+B^{21}&=0,\\
B^{22}&=0.
\end{align}
In other words, all the compatible infinitesimal deformations of $\Delta$ of degree 2 are parametrized by a single function $B(u)=B^{12}$ according to the given prescription.
Any of these compatible deformations can be obtained from $\Delta$ by the Miura transformation
\begin{equation}
 u\mapsto U=u-\epsilon  u_{\xi_1} \int^u B(s)\ud s
\end{equation}
which exists for any function $B$ of a single variable.
This means that $H^2_2=0$, in agreement with  formula~\eqref{HD2}.

\end{document}